\documentclass[a4paper,11pt]{article}

\usepackage[utf8]{inputenc}
\usepackage[english]{babel}
\usepackage{epic,eepic,graphicx}
\usepackage{enumerate}
\usepackage[all]{xy}
\usepackage{amsmath,amsfonts,amsthm,amssymb}
\usepackage{mathrsfs}
\usepackage{verbatim}
\usepackage{graphicx, wrapfig}
\usepackage{fixme}
\usepackage{enumitem}

\setenumerate[0]{label=(\roman*),font=\normalfont}

\newcommand{\N}{\mathbb{N}}
\newcommand{\Z}{\mathbb{Z}}

\newcommand{\C}{\mathbb{C}}

\newcommand{\e}{\varepsilon}

\newcommand{\fhi}{\varphi}

\newcommand{\norm}{\Arrowvert}
\newcommand{\ten}{\otimes}

\newcommand{\K}{\mathbb K}
\newcommand{\D}{\mathbb D}
\newcommand{\Cu}{\mathrm{Cu}}
\newcommand{\V}{\mathcal V}

\newcommand{\JS}{\mathcal Z}

\newcommand{\rank}{\textup{rank}}

\theoremstyle{definition} \newtheorem{mindef}{Definition}[section]

\newtheorem{remark}[mindef]{Remark}

\newtheorem{q}[mindef]{Question}

\theoremstyle{plain} \newtheorem{theo}[mindef]{Theorem}
\newtheorem{lemma}[mindef]{Lemma}
\newtheorem{cor}[mindef]{Corollary}
\newtheorem{prop}[mindef]{Proposition}

\author{Martin S. Christensen\footnote{This work was completed as a PhD-student at the University of Copenhagen}}
\date{}
\title{Regularity of Villadsen algebras and characters on their central sequence algebras}

\makeatletter
\renewcommand\@biblabel[1]{#1.}
\makeatother

\begin{document}

\maketitle

\begin{abstract}
We show that if $A$ is a simple Villadsen algebra of either the first type with seed space a finite dimensional CW complex, or of the second type, then $A$ absorbs the Jiang-Su algebra tensorially if and only if the central sequence algebra of $A$ does not admit characters. 

Additionally, in a joint appendix with Joan Bosa, we show that the Villadsen algebra of the second type with infinite stable rank fails the Corona Factorization Property, thus providing the first example of a unital, simple, separable and nuclear $C^\ast$-algebra with a unique tracial state which fails to have this property.
\end{abstract}

\section{Introduction}

Villadsen algebras, introduced by Villadsen in \cite{villadsen-perforation.} and \cite{villadsen-sr.}, respectively, fall into two types and both display properties not previously observed for simple AH algebras. Together they form a class of unital, simple and separable AH algebras exhibiting a wide range of exotic behaviour; arbitrary stable and real rank, arbitrary radius of comparison, and perforation in their ordered $K_0$ groups and Cuntz semigroups.

The first type of Villadsen algebras was introduced in \cite{villadsen-perforation.} as the first examples of unital, simple AH algebras with perforation in their ordered $K_0$ groups. In particular, they were the first examples of simple AH algebras without slow dimension growth. Modifying the construction, Toms exhibited for each positive real number $r>0$ a unital, simple AH algebra with rate of growth $r$ (in the sense that the radius of comparison is $r$); see \cite{toms-smooth.min.}. The techniques introduced by Villadsen also played a crucial role in Rørdam's construction in \cite{rordam-finite.and.infinite.projection.} of a simple, separable and nuclear $C^\ast$-algebra in the UCT class containing an infinite and a non-zero finite projection, the  first counterexample to the Elliott conjecture in its previous incarnation. In \cite{toms-the.classification.problem.} Toms used a modification of the AH algebras in \cite{villadsen-perforation.} to provide a particularly egregious counterexample to the previous Elliott conjecture. Toms and Winter gave a formal definition of Villadsen algebras of the first type in \cite{toms.winter-vi.algs.}, which includes Villadsen's original constructions, and the subsequent modifications of Toms in \cite{toms-the.classification.problem.} and \cite{toms-smooth.min.}. In the same paper they confirmed what has later been named the Toms--Winter conjecture for this class of $C^\ast$-algebras, i.e., they showed that for a simple Villadsen algebra of the first type with seed space a finite dimensional CW complex (see Definition \ref{def-villadsen.algebra.}), the regularity properties Jiang-Su stability, strict comparison of positive elements, and finite decomposition rank are equivalent. The latter regularity property, or even the weaker requirement of finite nuclear dimension, has since been proven to suffice for classification, under the additional assumption of UCT (the complete proof of this has a long history and is the work of many hands, but the final steps were carried out in \cite{gong.lin.niu-classification.of.finite.simple.algebras.},\cite{elliott.gong.lin.niu-classification.} and \cite{tikuisis.white.winter-quasidiagonality.}).

The second type of Villadsen algebras was introduced in \cite{villadsen-sr.} as the first examples of simple AH algebras with stable rank higher than one. In fact, every possible value of the stable rank is achieved, i.e., for each $1\leq k\leq \infty$ a unital, simple AH algebra $\mathcal V_k$ is constructed such that $\text{sr}(\mathcal V_k)=k+1$, and the real rank satisfies $k\leq RR(\mathcal V_k)\leq k+1$. In addition, each $C^\ast$-algebra $\mathcal V_k$ has a unique tracial state and perforation in the ordered $K_0$ group, in particular $\V_k\ten\JS\not\cong \V_k$. Ng and Kucerovsky showed in \cite{kucerosvky.ng-perforation.and.cfp.} that $\mathcal V_2$ has the Corona Factorization Property, thus providing the first example of a simple $C^\ast$-algebra satisfying this property while having perforation in the ordered $K_0$ group. The construction also formed the basis for Toms' counterexample to the previous Elliott conjecture in \cite{toms-counterexample.}.

As indicated in the preceding paragraphs, the class of Villadsen algebras form a rich class containing examples of both regular $C^\ast$-algebras and $C^\ast$-algebras displaying a wide range of irregularity, while still remaining amenable to analysis. As such, they form a good `test class' for statements concerning simple and nuclear $C^\ast$-algebras.

The central sequence algebra of a unital separable $C^\ast$-algebra $A$ (see Section \ref{sec-csa.} for a definition), which we denote $F(A)$, was studied extensively by Kirchberg in \cite{kirchberg-central.sequences.}, wherein the notation $F(A)$ was introduced, and the definition of $F(A)$ was extended to not necessarily unital $C^\ast$-algebras in a meaningful way (for instance, $F(A)$ is unital whenever $A$ is $\sigma$-unital, and the assignment $A\mapsto F(A)$ is a \textit{stable} invariant). In analogy with the \emph{von Neumann} central sequence algebra of $\mathrm{II}_1$-factors, the central sequence algebra detects absorption of certain well-behaved $C^\ast$-algebras. More precisely, if $B$ is a unital, separable $C^\ast$-algebra with approximately inner half-flip (i.e., the two factor embeddings $B\to B\ten B$ are approximately unitarily equivalent), then $A\ten B\cong A$ if there exists a unital embedding $B\to F(A)$. If, moreover, $B\cong \bigotimes_{n=1}^\infty B$, e.g., when $B$ is the Jiang-Su algebra $\mathcal Z$, then $A\ten B\cong A$ if \textit{and only if} such an embedding exists. Significant progress in our understanding of the central sequence algebra of stably finite $C^\ast$-algebras was obtained by Matui and Sato in \cite{matui.sato-strict.comparison., matui.sato-decomposition.rank.}. In these papers they introduced property (SI), a regularity property which facilitates liftings of certain properties of a tracial variant of the central sequence algebra to the central sequence algebra itself (see for instance \cite[Proposition 3.9]{kirchberg.rordam-char.}). Furthermore, they prove that whenever $A$ is a unital, simple, separable and nuclear $C^\ast$-algebra with strict comparison, then $A$ has property (SI) and as a consequence, if $A$ has only finitely many extremal tracial states, then $\JS$ embeds unitally in $F(A)$ hence $A\ten\JS\cong A$. 

Prompted by the analogy with von Neumann $\mathrm{II}_1$ factors one might hope that the McDuff dichotomy (see \cite{mcduff-central.sequences.}) carries over to $C^\ast$-algebras. However, as proven by Ando and Kirchberg in \cite{kirchberg.hiroshi-non.commutative.central.sequence.}, the central sequence algebra $F(A)$ is non-abelian whenever $A$ is separable and not type $\mathrm{I}$. In addition, it can happen that $F(A)$ is non-abelian and contains no simple, unital $C^\ast$-algebra other than $\C$ (see \cite[Corollary 3.14]{kirchberg-central.sequences.}). Hence, non-commutativity of $F(A)$ does not suffice to conclude regularity. Addressing this issue, Kirchberg and Rørdam asked the following question in \cite{kirchberg.rordam-char.}.
\begin{q}\label{central.question.}
Let $A$ be a unital and separable $C^\ast$-algebra. Does it follows that $A\ten\JS\cong A$ if and only if $F(A)$ has no characters?
\end{q}

Another question under consideration in the present paper is the following: given a unital, simple $C^\ast$-algebra $A$ with a unique tracial state, when can one conclude that $A$ is regular? In certain situations, a unique tracial state is sufficient to conclude regularity and even classifiability by the Elliott invariant. For instance, Elliott and Niu showed in \cite{elliott.niu-zero.mean.dimension.} that if $X$ is a compact metrizable Hausdorff space and $\sigma$ is a minimal homeomorphism of $X$ such that the dynamical system $(X,\sigma)$ is uniquely ergodic, i.e., $C(X)\rtimes_\sigma\Z$ has a unique tracial state, then $C(X)\rtimes_\sigma\Z$ is $\JS$-stable and classifiable (this is not automatic, see \ \cite{giol.kerr-perforation.}). Similarly, as proven by Niu (see \cite[Theorem 1.1]{niu-mean.dimension.}) if $A$ is a unital, simple AH algebra with diagonal maps such that the set of extremal tracial states is countable, then $A$ is without dimension growth. In particular, any AH algebra of this type with a unique tracial state has real rank zero (cf.\ \cite{bdr-real.rank.of.inductive.limits.}). On the other hand, as demonstrated in \cite{villadsen-sr.}, a unique tracial state does not suffice to conclude either real rank zero or $\mathcal Z$ stability for general AH algebras. It is therefore natural to ask what (if any) regularity properties are implied by the existence of a unique tracial state.

The Corona Factorization Property was introduced by Kucerovsky and Ng in \cite{kucerovsky.ng-the.cfp.} and is related to both the theory of extensions and the question of when extensions are automatically absorbing (see for instance \cite{kucerovsky.ng-the.cfp.and.approximate.unitary.equivalence.}). It is a very mild regularity condition, which nonetheless does exclude the most exotic behaviour. For instance, if $A$ is a separable $C^\ast$-algebra satisfying the Corona Factorization Property and $M_n(A)$ is stable for some $n\in\N$ then $A$ must also be stable (see \cite[Proposition 4.7]{opr-cfp.}). Under the additional assumption that $A$ is simple and has real rank zero it also follows that $A$ is either stably finite or purely infinite. Examples of $C^\ast$-algebras failing the Corona Factorization Property have been provided in the literature. For instance, the $C^\ast$-algebras constructed in \cite{rordam-finite.and.infinite.projection.} and \cite{rordam-stability.not.stable.property.} fail the Corona Factorization Property

The main result of the present paper is that question \ref{central.question.} has an affirmative answer when $A$ is either a simple Villadsen algebra of the first type with seed space a finite dimensional CW complex or a Villadsen algebra of the second type (see Theorem \ref{theo-VI.main.result.} and Corollary \ref{cor-type.II.characters.} respectively). 

Additionally, in a joint appendix with Joan Bosa, we show that the Villadsen algebra of the second type with infinite stable rank fails to have the Corona Factorization Property, thus providing an example of a unital, simple, separable and nuclear $C^\ast$-algebra with a unique tracial state which fails this property (see Theorem \ref{theo-non.cfp.}). While examples of unital, simple, separable and nuclear $C^\ast$-algebras without the Corona Factorization Property are already known, as noted above, the example provided here is to the best of the authors' knowledge the first of its kind with a unique tracial state. 

I thank the anonymous referee for several useful comments, which led to an improved exposition, and for pointing out an unclear point in my proof of Lemma \ref{lem-ch.char.ob.}. I also thank Mikael Rørdam for many helpful discussions of the present paper.

\section{Background}

\subsection{The Central Sequence Algebra}\label{sec-csa.}
Let $A$ be a unital $C^\ast$-algebra, $\omega$ be a free ultrafilter on $\N$ and $\ell^\infty(A)$ denote the sequences $(a_n)_n\subseteq A$ such that $\sup_n\| a_n\|<\infty$. The ultrapower $A_\omega$ of $A$ with respect to $\omega$ is defined by
\[
A_\omega:=\ell^\infty(A)/\{(a_n)_n\in \ell^\infty(A)\mid \lim_{n\to\omega}\|a_n\|=0\}.
\]  
Given a sequence $(a_n)_n\in \ell^\infty(A)$ let $[(a_n)_n]\in A_\omega$ denote the image under the quotient map. There is a natural embedding $\iota\colon A\to A_\omega$ given by 
\(\iota(a)=[(a,a,a,\dots)]\). Since $\iota$ is injective it is often suppressed and $A$ is considered to be a subalgebra of $A_\omega$, a convention we shall follow here. The central sequence algebra $F(A)$ of $A$ is defined by \( F(A):=A_\omega\cap A' \). The notation $F(A)$ was introduced by Kirchberg in \cite{kirchberg-central.sequences.}, wherein the definition of the central sequence algebra was extended to (possibly non-unital) $\sigma$-unital $C^\ast$-algebras in a meaningful way. We retain this notation, although only unital $C^\ast$-algebras are considered here, to emphasize the connection with Kirchberg's work. Furthermore, the ultrafilter is suppressed in the notation, since the isomorphism class of (unital) separable sub-$C^\ast$-algebras $B\subseteq F(A)$ is independent of the choice of free ultrafilter. More precisely, if $B$ is a separable $C^\ast$-algebra and there exists a (unital) injective ${}^\ast$-homomorphism $B\to A_\omega\cap A'$ for some free ultrafilter $\omega$ on $\N$, then there exists a (unital) injective ${}^\ast$-homomorphism $B\to A_{\omega'}\cap A'$ for any other free ultrafilter $\omega'$ on $\N$. In particular, the question of whether $F(A)$ has characters is independent of the choice of free ultrafilter (see \cite[Lemma 3.5]{kirchberg.rordam-char.}). Whether $A_\omega\cap A'\cong A_{\omega'}\cap A'$ for arbitrary free ultrafilters $\omega$ and $\omega'$ on $\N$ depends on the Continuum Hypothesis (see \cite{ge.hadwin-ultraproducts.} and \cite[Theorem 5.1]{farah.hart.sherman-model.theory.I.}).

As described in \cite{kirchberg.rordam-char.}, building on results from \cite{robert.rordam-div.}, there is a useful relationship between divisibility properties of $F(A)$ and comparability properties of $\Cu(A)$. We rely on an elaboration of this technique to obtain our results.

%
%
%
%

\subsection{Vector Bundles and Characteristic Classes}
Readers who are unfamiliar with the theory of characteristic classes of (complex) vector bundles may wish to consult \cite{milnor+stasheff} for a general textbook on the subject. Alternatively, the papers \cite{rordam-finite.and.infinite.projection.} and \cite{villadsen-perforation.} also contains good summaries of (the relevant parts of) the theory.

In order to access the machinery of characteristic classes within the framework of $C^\ast$-algebras we need the following observation: Let $\K$ denote the compact operators acting on a separable, infinite-dimensional Hilbert space $\mathcal H$, let $p\in C(X)\ten\K$ be a projection and let $\xi_p$ denote vector bundle over $X$ given by
$$\xi_p:=\{(x,v)\in X\times\mathcal H\mid v\in p(x)(\mathcal H)\}.$$
It is a consequence of Swan's Theorem that the assignment $p\mapsto\xi_p$ induces a one-to-one correspondence of Murray-von Neumann equivalence classes of projections in $C(X)\ten\K$ with isomorphism classes of
vector bundles over $X$, in such a way that $q\precsim p$ if and only if there
exists a vector bundle $\eta$ over $X$ such that $\xi_q\oplus\eta\cong \xi_p$.
We shall be concerned with the ordering of vector bundles  according to
the above described pre-order. For this purpose we employ the machinery of characteristic classes of vector bundles described below, a technique pioneered by Villadsen in \cite{villadsen-perforation.} and \cite{villadsen-sr.}.

Given a compact Hausdorff space $X$ and vector bundle $\omega$ of (complex) fibre
dimension $k$, the \emph{total Chern class} $c(\omega)\in H^*(X)$ is 
$$c(\omega)=1+\sum_{i=1}^\infty c_i(\omega),$$
where $c_j(\omega)\in H^{2j}(X)$ is the $j$'th Chern class for each $1\leq j\leq k$, and $c_j(\omega)=0$ whenever $j>k$. 
Furthermore, the top Chern class $c_k(\omega)$ is the Euler class $e(\omega)$ of $\omega$. We will simply refer to $c(\omega)$ as the Chern class of $\omega$, rather than the \emph{total} Chern class.
The Chern class has the following properties:
\begin{enumerate}
\item If $\theta_k$ denotes the trivial vector bundle of fibre dimension
$k\in\N$, then $c(\theta_k)=1\in H^0(X)$ for any $k\in\N$.
\item For arbitrary vector bundles $\omega,\eta$ over $X$ we have \(c(\omega\oplus\eta)=c(\omega)c(\eta)\),
where the product is the cup product in the cohomology ring $H^*(X)$.
\item If $Y$ is another compact Hausdorff space and $f\colon Y\to X$ is continuous
then \(c(f^*(\omega))=f^*(c(\omega)) \).
\end{enumerate}
Properties (ii) and (iii) above also holds for the Euler class, while
the first property instead becomes $e(\theta_k)=0$ for all $k\in\N$. This can
be deduced from the above description of the Chern class.

In the following sections it will suffice to find a reasonably good method for determining
which Chern classes of a vector bundle are non-zero. Such a method is provided by the following observation. Given a finite number of sets $X_1,\dots,X_n$, let $\rho_j\colon X_1\times\cdots\times X_n\to X_j$ denote the $j$'th coordinate projection. If each of the spaces $X_1,\dots,X_n$ is a finite CW-complex such that $H^i(X_j)$ is a free $\Z$-module for each $i$ and $j$, it follows from the Künneth formula (see \cite[Theorem A.6]{milnor+stasheff}) that the map 
$$\mu\colon H^{i_1}(X_1)\ten H^{i_2}(X_2)\ten\cdots \ten H^{i_n}(X_n)\to H^{i}(X_1\times X_2\times\cdots\times X_n),$$
where $i=\sum_{k=1}^n i_k$, given by
$$a_1\ten a_2\ten\cdots\ten a_n\mapsto \rho_1^\ast(a_1)\rho_2^\ast(a_2)\cdots\rho_n^\ast(a_n),$$
is injective. A particular application of this observation is the following: suppose that $X_1,\dots,X_n$ satisfies the hypothesis above and,
for each $i=1,\dots,n$, that $\xi_i$ is a vector bundle over $X_i$ 
such that $e(\xi_i)\in H^*(X_i)$ is non-zero for $i=1,\dots,n$. Since each $H^i(X_j)$ is without torsion, the element $e(\xi_1)\ten\cdots\ten e(\xi_n)$ is also non-zero, whence it follows from naturality of the Euler class and the product formula above that
\begin{align*}
e\big(\rho_1^*(\xi_1)\oplus\rho_2^*(\xi_2)\oplus\cdots\oplus\rho_n^*(\xi_n)\big)
&=\rho_1^\ast(e(\xi_1)) \rho^\ast_2(e(\xi_2))\cdots\rho^\ast_n(e(\xi_n))\\
&=\mu(e(\xi_1)\ten\cdots\ten e(\xi_n))\neq 0.
\end{align*}
We will apply this observation only to the situation where each $X_i$ is either of the form $(S^2)^k$ for some $k$ or a complex projective space $\C P^k$, in which case the hypothesis' are satisfied.

%
%
%
%

\subsection{The Cuntz Semigroup, Comparison and Divisibility}
We give a brief introduction to the Cuntz semigroup as defined in \cite{CEI-the.cuntz.semigroup.}. We restrict our attention to the properties needed in the current exposition, and interested readers should consult \cite{CEI-the.cuntz.semigroup.} or \cite{APT-the.cuntz.semigroup.} for a fuller exposition.

Let $A$ be a $C^\ast$-algebra and let $a,b\in A_+$. We say that $a$ is Cuntz
dominated by $b$, and write $a\precsim b$, if there exists a sequence
$(x_n)_n\subseteq A$ such that \(\| a-x_n^*bx_n\|\to 0 \).
We say that $a$ is Cuntz equivalent to $b$, and write $a\sim b$, if $a\precsim
b$ and $b\precsim a$. Let $\mathbb K$ denote the compact operators on
a separable, infinite-dimensional Hilbert space and define
\[
\Cu(A):=(A\ten\K)_+/\!\sim.
\]
We write $\langle a\rangle$ for the equivalence class of an element $a\in
(A\ten\K)_+$. Then $\Cu(A)$ becomes an ordered abelian semgroup when equipped
with the operation
\[
\langle a\rangle+\langle b\rangle:=\langle a\oplus b\rangle, \quad a,b\in
(A\ten\K)_+
\]
and order defined by $\langle a\rangle \leq \langle b\rangle$ if and only if
$a\precsim b$. Additionally, any upwards directed countable set $S\subseteq \Cu(A)$ admits a supremum. Given $x,y\in\Cu(A)$ we say that $x$ is \emph{compactly contained} in $y$, and write $x\ll y$, if for any increasing sequence $(y_k)_k\subseteq\Cu(A)$ with $\sup_k y_k=y$ there exists $k_0\in\N$ such that $x\leq y_{k_0}$. Equivalently, if $a,b\in (A\ten\K)_+$ then $\langle a\rangle \ll \langle b\rangle$ if and only if there exists $\e>0$ such that $a\precsim (b-\e)_+$. An element $x\in\Cu(A)$ satisfying $x\ll x$ is said to be \textit{compact}. Note that $\langle p\rangle$ is compact whenever $p\in (A\ten\K)_+$ is a projection.

The following proposition is a strengthening of \cite[Theorem 4.9]{kirchberg.rordam-char.} with essentially the same proof. Although the strengthening is minor, it is crucial to Theorem \ref{theo-VI.main.result.} and Corollary \ref{cor-type.II.characters.}. 
\begin{prop}\label{cor-skarp.kr.}
Let $A$ be a unital, separable $C^\ast$-algebra. If $F(A)$ has no characters,
then for each $m\in\N$ there exists $n\in\N$ such that the following
holds: given $x,y_1,\dots,y_n\in \Cu(A)$ such that $x\leq my_i$ for all $i=1,\dots,n$, then
\(x\leq\sum_{i=1}^n y_i\).
\end{prop}
\begin{proof}
It follows from \cite[Lemma 3.5]{kirchberg.rordam-char.} that there exists a unital,
separable sub-$C^\ast$-algebra $B\subseteq F(A)$ such that $B$ has no
characters. Hence, \cite[Corollary 5.6 (i) and Lemma 6.2]{robert.rordam-div.}
imply that for each $m\in\N$ there exists $n\in\N$ such that the infinite
maximal tensor product $C^\ast$-algebra \(D:=\bigotimes_{k\in\N}B\)
is weakly $(m,n)$-divisible, i.e., there exist elements $y_1,\dots,y_n\in\Cu(D)$ satisfying $my_i\leq \langle 1_D\rangle$, for all $i=1,\dots,n$, and $\langle 1_D\rangle\leq \sum_{j=1}^ny_j$. Note that since $B\subseteq F(A)$ is unital and separable, it
follows from \cite[Corollary 1.13]{kirchberg-central.sequences.} that there exists a unital ${}^*$-homomorphism
$\fhi:D\to F(A)$. Let $P\subseteq A_\omega$ denote the image under the natural
map $A\ten_{\max}D\to A_\omega$. By \cite[Lemma 4.1]{kirchberg.rordam-char.}
the induced map $\Cu(A)\to\Cu(P)$ is an order embedding, and therefore the result finally follows from \cite[Lemma
6.1]{robert.rordam-div.}.
\end{proof}

%
%
%
\section{Villadsen Algebras of the first type}\label{sec-the.first.type.}
In this section we study Villadsen algebras of the first type, as defined by Toms and Winter in \cite{toms.winter-vi.algs.} based on the construction by Villadsen in \cite{villadsen-perforation.}. We prove that for a simple Villadsen algebra $A$ of the first type with seed space a finite dimensional CW complex, $F(A)$ has no characters if and only if $A$ has strict comparison of positive elements (Theorem \ref{theo-VI.main.result.}). We also note in passing that if $A$ is not an AF algebra, then $A$ has real rank zero if and only if it has a unique tracial state (Proposition \ref{prop-real.rank.and.unique.trace.}).

For the readers convenience we recall the definition of a Villadsen algebra of the first type (see also \cite{toms.winter-vi.algs.}).
\begin{mindef}
Let $X, Y$ be a compact Hausdorff spaces and $n,m\in\N$ be given such that
$n\mid m$. A ${}^*$-homomorphism $\fhi:M_n\ten C(X)\to M_m\ten C(Y)$ is said to be 
\textbf{diagonal} if it has the form
\[
f\mapsto \begin{pmatrix}
f\circ \lambda_1 & 0 				& \cdots & 0\\
0 				 & f\circ\lambda_2  & 		 & \vdots \\
\vdots 			 &					& \ddots &  0\\
0 				 &	\cdots			& 0 	 & f\circ\lambda_{m/n}
\end{pmatrix},
\]
where each $\lambda_i\colon Y\to X$ is a continuous map for $i=1,\dots,m/n$. The
maps $\lambda_1,\dots,\lambda_{m/n}$ are called the \textbf{eigenvalue maps} of $\fhi$.

The map $\fhi$ above is said be a \textbf{Villadsen map of the first type}
(a $\mathcal V$I-map) if $Y=X^k$ for some $k\in\N$ and each eigenvalue map is
either a coordinate projection or constant.
\end{mindef}
Note that, in contrast with the construction in \cite{villadsen-perforation.}, given a $\V$I map $\fhi\colon C(X)\ten M_n\to C(X^k)\ten M_m$ as above, it is not necessary that the coordinate projections that occur as eigenvalue maps for $\fhi$ are distinct, nor that every possible coordinate projection $X^k\to X$ occurs as an eigenvalue map for $\fhi$.
\begin{mindef}\label{def-villadsen.algebra.}
Let $X$ be a compact Hausdorff space and let
$(n_i)_{i\in\N}$ and $(m_i)_{i\in\N}$ be sequences of natural numbers with $n_1=1$ and such that $m_i|m_{i+1}$ and $n_i|n_{i+1}$ for all $i\in\N$. Put
$X_i=X^{n_i}$. A unital $C^\ast$-algebra $A$ is said to be a
\textbf{Villadsen algebra of the first type} (a $\mathcal V$I algebra) if it
can be written as an inductive limit
\begin{align*}
A\cong \varinjlim(M_{m_i}\ten C(X_i),\fhi_i)\label{standard-decomposition},
\end{align*}
where each $\fhi_i$ is a $\mathcal V$I map. We refer to the above inductive
system as a \textbf{standard decomposition} for $A$ with \textbf{seed space}
$X$.
\end{mindef}
Although not required in the above definition, we shall only consider \emph{simple} $\V$I algebras in the present paper. Additionally, we require that the seed space is a finite-dimensional $CW$ complex. This is a particularly tractable class of $C^\ast$-algebras, as demonstrated by the following theorem due to Toms and Winter.

\begin{theo}[See \cite{toms.winter-vi.algs.}]\label{theo-toms.winter.}
Let $A$ be a simple $\mathcal V$I algebra admitting a standard decomposition with seed space a finite-dimensional $CW$ complex. The following are equivalent:
\begin{enumerate}
\item $A$ has finite decomposition rank.
\item $A$ is $\mathcal Z$-stable.
\item $A$ has strict comparison of positive elements.
\item $A$ has slow dimension growth as an AH algebra.
\end{enumerate}
\end{theo}

It follows directly from Definition \ref{def-villadsen.algebra.} that if $X$ is a zero-dimensional CW complex, i.e., is a finite discrete space, then the corresponding $\V$I algebra is a unital AF algebra. In the interest of the fluency of this exposition we shall henceforth assume that $\text{dim}(X)>0$, since the case $\text{dim}(X)=0$ often requires separate consideration, and unital, simple AF algebras are already well-understood. We proceed to introduce some notation.

For each $j\geq i$ let $\pi_{i,j}^{(s)}$ denote the $s$'th coordinate projection $X_j=X_i^{(n_j/n_i)} \to X_i$. Following standard notation, we set \( \fhi_{i,j}:=\fhi_{j-1}\circ\cdots\circ\fhi_i\),
when $j>i$, set $\fhi_{i,i}$ to be the identity map on $M_{m_i}\ten C(X_i)$,
and $\fhi_{i,j}$ to be the zero map when $j<i$. It is easy to check that
$\fhi_{i,j}\colon M_{m_i}\ten C(X_i)\to M_{m_j}\ten C(X_j)$ is a $\mathcal V$I map 
whenever $j> i$. For each $j> i$ let $E_{i,j}$ denote the set of 
eigenvalue maps of $\fhi_{i,j}$, and for each $\lambda\in E_{i,j}$ let $m(\lambda)$ denote the multiplicity of $\lambda$, i.e., the number of times $\lambda$ occurs as an eigenvalue map of $\fhi_{i,j}$. Furthermore, let
\begin{align*}
E^{(1)}_{i,j}&:=\{\lambda\in E_{i,j}\mid \lambda\text{ is a coordinate
projection}\},\\
E^{(2)}_{i,j}&:=\{\lambda\in E_{i,j}\mid \lambda\text{ is constant}\}.
\end{align*}
We will refer to the eigenvalue maps $\lambda\in E^{(2)}_{i,j}$ as point evaluations. For each $i<j$ write
\( \fhi_{i,j}=\psi_{i,j}\oplus\chi_{i,j}\), where $\psi_{i,j}$ is the diagonal ${}^\ast$-homomorphism corresponding to the eigenvalue maps of
$\fhi_{i,j}$, which are contained in $E^{(1)}_{i,j}$, and $\chi_{i,j}$ is the
diagonal ${}^\ast$-homomorphism corresponding to the eigenvalue maps of $\fhi_{i,j}$, which
are contained in $E^{(2)}_{i,j}$. Finally, we define the following numbers
\begin{align*}
N(i,j):=|E^{(1)}_{i,j}|,\quad \alpha(i,j):=\sum_{\lambda\in E^{(1)}_{i,j}}m(\lambda),\quad
M(i,j):=\sum_{\lambda\in E_{i,j}}m(\lambda).
\end{align*}
In other words, $M(i,j)$ denotes the multiplicity (number of eigenvalue maps)
of $\fhi_{i,j}$, $\alpha(i,j)$ denotes the number of coordinate projections
occurring in $\fhi_{i,j}$, while $N(i,j)$ denotes the number of \textit{different}
coordinate projections occurring in $\fhi_{i,j}$. Note that when $j>i$ we
have
\[
M(i,j)=M(i,j-1)M(j-1,j),\quad
N(i,j)=N(i,j-1)N(j-1,j),\]
\[\alpha(i,j)=\alpha(i,j-1)\alpha(j-1,j),\]
and that \( 0\leq \frac{N(i,j)}{M(i,j)}\leq \frac{\alpha(i,j)}{M(i,j)}\leq 1\).
In particular, the sequences
\[
\left(\frac{N(i,j)}{M(i,j)}\right)_{j>i}\quad \text{and} \quad
\left(\frac{\alpha(i,j)}{M(i,j)}\right)_{j>i}
\]
are decreasing and convergent. Furthermore, setting \( c_i=\lim_{j\to\infty}\frac{N(i,j)}{M(i,j)}\) and \(d_i=\lim_{j\to\infty}\frac{\alpha(i,j)}{M(i,j)}\), the sequences $(c_i)_i$ and $(d_i)_i$ are both increasing and $c_i\leq d_i$ for all $i\in\N$. In fact, it is easy to check that either $c_i=0$ for all $i\in\N$ or $\lim_{i\to\infty}c_i=1$. Similarly, either $d_i=0$ for all $i$ or $\lim_{i\to\infty}d_i=1$ (see the proof of \cite[Lemma 5.1]{toms.winter-vi.algs.}).

During the proof of Theorem \ref{theo-VI.main.result.} we need the following Chern class obstruction, essentially due to Villadsen, and later refined by Toms in \cite{toms-the.classification.problem.},\cite{toms-smooth.min.} and Toms--Winter in \cite{toms.winter-vi.algs.}. In the statement (and proof) of the lemma, we will use the following notation: given a finite cartesian power of spheres $(S^2)^n$, and $1\leq j \leq n$, let $\rho_j\colon (S^2)^n\to S^2$ denote the $j$'th coordinate projection.

%
%
%
%
\begin{lemma}\label{lem-ch.char.ob.}
Let $A$ be a Villadsen algebra which admits a standard decomposition
$(A_i,\fhi_i)$ with seed space a finite-dimensional $CW$-complex $X$ of
non-zero dimension. Furthermore, assume that, for some $i\in\N$, there exist $n\in\N$, a closed subset $X_i\supseteq K\cong (S^2)^n$ and a positive element $a\in A_i\ten \K$,  such that $a|_K$ is a projection for which the corresponding vector bundle $\xi$ is of the form $\xi\cong \rho_1^\ast(\eta)\oplus\cdots\oplus\rho_n^\ast(\eta)$, where $\eta$ is a (complex) line bundle over $S^2$ with 
non-zero Euler class $e(\eta)$. For each $j>i$ define a closed subset 
$K_{i,j}\subseteq X_j$ by
$$K_{i,j}:=\times_{s=1}^{n_j/n_i}K_{i,j}^{(s)},$$
where
$$K_{i,j}^{(s)}=\begin{cases}
	K, & \text{if }\pi_{i,j}^{(s)}\in E_{i,j}^{(1)},\\
	\{x_{j}\}, & \text{otherwise}.
\end{cases}$$
and $x_j\in X_i$. Let $\xi_j$ denote the vector bundle over $K_{i,j}$ corresponding to $\psi_{i,j}(a)|_{K_{i,j}}$. Then the $nN(i,j)$'th Chern class $c_{nN(i,j)}(\xi_j)$ is non-zero.
\end{lemma}
\begin{proof}
Note that $K_{i,j}\cong K^{ N(i,j)}\cong (S^2)^{ nN(i,j)}$. Since $a|_K$ is a projection, it follows
from the definition of $\psi_{i,j}$, that $\psi_{i,j}(a)|_{K_{i,j}}$ is a projection. As in the statement above, let $\xi$ denote the vector bundle corresponding to $a|_K$ and $\xi_j$ the vector bundle corresponding to $\psi_{i,j}(a)|_{K_{i,j}}$. We easily deduce that
\begin{align*}
\xi_j\cong \bigoplus_{\lambda\in E_{i,j}^{(1)}}\bigoplus_{j=1}^{m(\lambda)}\lambda^*(\xi).
\end{align*}
Applying the Chern class to this equation, and using the product formula, we obtain 
\begin{align*}
c(\xi_j)&=\prod_{\lambda\in E_{i,j}^{(1)}}\prod_{j=1}^{m(\lambda)}
c(\lambda^*(\xi))=\prod_{\lambda\in E_{i,j}^{(1)}}\lambda^*\big(c(\xi)\big)^{m(\lambda)}.
\end{align*}
Write $E_{i,j}^{(1)}=\{\lambda_1,\lambda_2,\dots,\lambda_{N(i,j)}\}$. 
For $l=1,\dots,N(i,j)$ and $s=1,\dots,n$ set $z_{l,s}:=\lambda_l^*\big(\rho_s^\ast(e(\eta))\big)$. Since $e(\eta)^2=0$ (recall that $H^j(S^2)=0$ for all $j>2$), we find that $z_{l,s}^m=0$ for $l,s$ and $m>1$. By assumption, $\xi\cong\rho_1^\ast(\eta)\oplus\cdots\oplus \rho_n^\ast(\eta)$, whence
\begin{align*}
c(\xi_j)&=\prod_{l=1}^{N(i,j)}\prod_{s=1}^n (1+z_{l,s})^{m(\lambda_l)}=\prod_{l=1}^{N(i,j)}\prod_{s=1}^n (1+m(\lambda_l)z_{l,s}).
\end{align*}
Given a subset $S\subseteq \{1,\dots,n\}$ let
$z_{l,S}:= \prod_{s\in S}m(\lambda_l)z_{l,s}$ when $S\neq \emptyset$ and $z_{l,\emptyset}:= 1$ for all $1\leq l\leq N(i,j)$. It follows from the above computation that, for $1<q\leq \text{rank}(\xi_j)$, the $q$'th Chern class
$c_q(\xi_j)$ can be computed as \( \sum \prod_{l=1}^{N(i,j)}z_{l,S_l}\),
where the sum ranges over all families $\{S_l\}_l$ of subsets $S_l\subseteq \{1,\dots,n\}$ such that \(\sum_{l=1}^{N(i,j)}|S_l|=q\). Now, supposing that $\{S_l\}_l$ is a family of subsets $S_l\subseteq\{1,\dots,n\}$ such that $S_{l_0}\neq \{1,\dots,n\}$ for some $l_0$, it follows that $\sum_{l=1}^{N(i,j)}|S_l|<nN(i,j)$. In particular, we find that
\[
c_{nN(i,j)}(\xi_j)=\prod_{l=1}^{N(i,j)}z_{l,\{1,\dots,n\}}=\prod_{l=1}^{N(i,j)}\prod_{s=1}^n m(\lambda_l) z_{l,s}.
\]
It therefore follows from the Künneth formula that $c_{nN(i,j)}(\xi_j)\neq 0$.
\end{proof}

%
%

The following theorem is the main result of this section. The proof is based on the proof of \cite[Lemma 4.1]{toms.winter-vi.algs.}. However, since the statement of the following theorem is different, the proof needs to be modified, and in the interest of clarity of the exposition, we include a full proof.

\begin{theo}\label{theo-VI.main.result.}
Let $A$ be a simple Villadsen algebra of the first type 
which admits a standard decomposition
$(A_i,\fhi_i)$ with seed space a finite-dimensional $CW$-complex. 
Then $A$ has strict comparison (and hence $A\ten\JS\cong A$) if and only if $F(A)$ has no characters.
\end{theo}

\begin{proof}
Assume $A$ has strict comparison. Then it follows from Theorem \ref{theo-toms.winter.} that $A\ten\JS\cong A$, whence there exists a unital embedding $\JS\to F(A)$. Since $\JS$ has no characters it follows that $F(A)$ does not admit a character either. We show, using Proposition \ref{cor-skarp.kr.}, that $F(A)$ has at least one character if $A$ does not have strict comparison.

Fix $n\geq 2$. Since $A$ does not have strict comparison it follows from \cite[Lemma 5.1]{toms.winter-vi.algs.} that 
\begin{align}
\lim_{i\to\infty}\lim_{j\to\infty}\frac{N(i,j)}{M(i,j)}=1.\label{ratio-assumption}
\end{align}
Note that since $\text{dim}(X)>0$ and $A$ is simple, the number of point evaluations occurring as eigenvalue maps
in $\fhi_{i,j}$ is unbounded as $j\to\infty$ for any $i\in\N$. In particular, $M(i,j)\to\infty$ as $j\to\infty$, whence \eqref{ratio-assumption} implies $\text{dim}(X_i)\to\infty$ as
$i\to\infty$. Hence, we may choose $i\in\N$ such that 
$\text{dim}(X_i)\geq 3n$ and
\begin{align}
\frac{N(i,j)}{M(i,j)}\geq\frac{2n-1}{2n},\quad \text{for all }j>i.\label{eq-ratio.assumption.1.}
\end{align}
Choose an open subset $O\subseteq X_i$ such that $O\cong(-1,1)^{\text{dim}(X_i)}=:D$. Let 
$$\overline Y:=\{x\in(-1,1)^3\mid \text{dist}\big(x,(0,0,0)\big)=1/2\}$$
and
$$\overline Z:=\{x\in (-1,1)^3\mid 1/3\leq\text{dist}\big(x,(0,0,0)\big)\leq2/3\}.$$
Furthermore, define closed subsets
$$K:=\overline Y^{\times n}\times \{0\}^{\text{dim}(X_i)-3n}\subseteq D$$
and
$$Z:=\overline Z^{\times n}\times [-4/5,4/5]^{\text{dim}(X_i)-3n}\subseteq D.$$
Let $Z_0$ denote the interior of $Z$ and note that $K\subseteq Z_0$. 
We identify $K$ and $Z$ with their homeomorphic images in $X_i$ and note that
$K\cong (S^2)^n$. For each $l=1,\dots, n$, let $\rho_l:(S^2)^n\to S^2$ denote
the $l$'th coordinate projection. Choose some line bundle $\eta$ over $S^2$
with non-zero Euler class $e(\eta)$ (for
instance the Hopf bundle), and set $\eta_l:=\rho_l^*(\eta)$.
We consider each $\eta_l$ to be a
vector bundle over $K$. Furthermore, let $\theta_2$ denote the trivial
vector bundle of fibre dimension $2$ over $K$. It follows from
\cite[Proposition 9.1.2]{husemoller} that $\theta_2\precsim \eta_l\oplus\eta_l\oplus \eta_l$, for each $l=1,\dots,n$, while 
\( \theta_2\not\precsim \bigoplus_{l=1}^n\eta_l\), since the Euler class of the right hand vector bundle is non-zero. We
aim to construct positive elements in $A$ such that the above relationships
between vector bundles persist in $\Cu(A)$.

Let $\mathrm{pr}\colon\overline Z\to \overline Y$ be the projection along
rays emanating from the origin and let $f\colon X_i\to \C$ be a continuous map
satisfying $f|_K\equiv 1$ and $f|_{X_i\backslash Z_0}\equiv 0$. Let $P\colon Z\to K$ 
be given by
\[
P=\underbrace{\mathrm{pr}\times\cdots\times \mathrm{pr}}_{n\text{
times}}\times\underbrace{\text{ev}_0\times\cdots\times\text{ev}_0}_{\text{dim}(X_i)-3n\text{ times}},
\]
where $\text{ev}_0(z)=0$ for any $z\in (-1,1)$. For each $l=1,\dots,n$, let $p_l\in C(Z,\K)$ denote the projection corresponding to $P^\ast(\eta_l)$ and let $p'\in C(Z,\K)$ denote the projection corresponding to $P^\ast(\theta_2)$. Define elements 
$b_l,a\in A_i$, for $l=1,\dots,n$, by $b_l:= f\cdot p_l$ and $a:= f\cdot p'$. Since $f\in A_i$ is central, and $p'\precsim p_l\oplus p_l\oplus p_l$ for each $l=1,\dots,n$, it easily follows that $ a\precsim b_l\oplus b_l\oplus b_l$, for each $l=1,\dots,n$. Let
\begin{align*}
x:=\langle\fhi_{i,\infty}( a)\rangle \in \Cu(A),\quad
y_l:=\langle\fhi_{i,\infty}( b_l)\rangle\in \Cu(A),\text{ for } l=1,\dots,n.
\end{align*}
Clearly  $x\leq 3y_l$ for $l=1,\dots,n$. To finish the proof we need to show $x\not\leq y_1+y_2+\cdots+y_n$, and then Proposition \ref{cor-skarp.kr.} (with $m=3$) will yield the desired result.

Letting $a$ be given as above and $b=\bigoplus_{l=1}^n b_l\in (A_i\ten\K)_+$, we aim to show that $\fhi_{i,\infty}(a)\not\precsim \fhi_{i,\infty}(b)$ in $A\ten\K$. 
It suffices to prove that
$$\norm v^*\fhi_{i,j}(b)v-\fhi_{i,j}(a)\norm\geq\frac 12,$$
for each  $j>i$ and $v\in A_j\ten\K$. Note that $\chi_{i,j}(b)$ is a constant, positive, matrix valued function, whence
$q:=\lim_{n\to\infty}\chi_{i,j}(b)^{1/n}\in A_j\ten \K$ is a constant projection such that $\chi_{i,j}(b)q=\chi_{i,j}(b)$. Setting $Q:=\psi_{i,j}(\mathbf 1)\oplus\chi_{i,j}(b)^{1/2}$, we have
\begin{align}
\fhi_{i,j}(b)=\psi_{i,j}(b)\oplus\chi_{i,j}(b)=Q(\psi_{i,j}(b)\oplus q)Q.\label{eq-rewriting.}
\end{align}
Now, let $j>i$ be given and suppose for a contradiction, that there exists $v\in A_j\ten\K$ such that $\norm v^*\fhi_{i,j}(b)v-\fhi_{i,j}(a)\norm<1/2$.
Then, setting $w:=Qv\psi_{i,j}(\mathbf 1_{A_i})$, it follows from \eqref{eq-rewriting.} that
\begin{align}
\frac 12>\norm v^*Q(\psi_{i,j}(b)\oplus q)Qv-\fhi_{i,j}(a)\norm\geq \norm w^*(\psi_{i,j}(b)\oplus q)w-\psi_{i,j}(a)\norm.\label{essen-modstr}
\end{align}
This estimate remains valid upon restriction to any closed subset of $X_j$.

Let $\xi$ denote the vector bundle over K corresponding to
$b|_K$. Plug $A$, $X$, $X_i$, $b$, $K$ and $\xi$ into Lemma \ref{lem-ch.char.ob.} to get $K_{i,j}\subseteq X_j$ and $\xi_j$. Note that $b|_K=(b_1|_K)\oplus\cdots\oplus (b_n|_K)$, whence $\xi\cong \rho_1^\ast(\eta)\oplus\cdots\oplus \rho_n^\ast(\eta)$, and therefore the hypothesis of Lemma \ref{lem-ch.char.ob.} are satisfied. It is easily deduced that $q|_{K_{i,j}}$
corresponds to a trivial vector bundle $\theta_{nr}$, where $0\leq r\leq M(i,j)-\alpha(i,j)$, and since $a|_K\in C(K)\ten\K$ is a constant projection valued function of rank $2$ it follows that $\psi_{i,j}(a)|_{K_{i,j}}$ corresponds to the trivial vector bundle $\theta_{2\alpha(i,j)}$.
It therefore follows from \eqref{essen-modstr} and \cite[Lemma 2.1]{toms-the.classification.problem.} that
there exists a vector bundle $\zeta$ of fibre dimension $(n-2)\alpha(i,j)+nr$ and $t\in\N$ such that
\[
\zeta\oplus\theta_{2\alpha(i,j)+t}\cong \xi_j\oplus\theta_{nr+t}.
\]
Applying the Chern class to both sides of the above expression, we obtain
that \( c(\zeta)=c(\xi_j) \).
In particular, $c_{nN(i,j)}(\zeta)=c_{nN(i,j)}(\xi_j)$, whence Lemma \ref{lem-ch.char.ob.} implies that $c_{nN(i,j)}(\zeta)$ is non-zero. Hence $\text{rank}(\zeta)\geq nN(i,j)$, and therefore
\begin{align*}
nN(i,j)&\leq (n-2)\alpha(i,j)+nr\\
&\leq (n-2)\alpha(i,j)+n(M(i,j)-\alpha(i,j))\\
&\leq nM(i,j)-2N(i,j).
\end{align*}
Thus, dividing both sides by $nM(i,j)$ we obtain
$$\frac{N(i,j)}{M(i,j)}\leq 1-\frac 2n\cdot \frac{N(i,j)}{M(i,j)}.$$
Hence \eqref{eq-ratio.assumption.1.} implies
\begin{align*}
\frac{2n-1}{2n}\leq 1-\frac{2(2n-1)}{n(2n)}=\Big(\frac{n-1}n\Big)^2<\frac{n-1}{n},
\end{align*}
which is the desired contradiction.
\end{proof}

Before considering Villadsen algebras of the second type let us record the following proposition, which is an aggregation of results by other authors. However, it does serve to illustrate the added complexity of Villadsen algebras of the second type (compare with Theorem \ref{theo-villadsens.resultater.}), which are less studied than those of the first type.
\begin{prop}\label{prop-real.rank.and.unique.trace.}
Suppose $A$ is a simple Villadsen algebra which admits a standard decomposition with seed space a finite dimensional $CW$-complex of non-zero dimension. Then $A$ has real rank zero if and only if $A$ has a unique tracial state.
Furthermore, in this case, $A\ten\mathcal Z\cong A$.
\end{prop}

\begin{proof}
The proof that real rank zero implies unique tracial state is essentially contained in \cite[Proposition 7.1]{toms.winter-vi.algs.}. Indeed, replacing every instance of $N(i,j)$ in the cited proof with $\alpha(i,j)$, it follows that if $RR(A)=0$, then \( \lim_{j\to\infty}\frac{\alpha(i,j)}{M(i,j)}=0\) for all $i\in\N$. It is easy to check that this implies that $A$ has a unique tracial state. Furthermore, the statement that $A$ is $\JS$ stable follows from \cite[Proposition 7.1]{toms.winter-vi.algs.} and  a series of results summarized in \cite[Theorem 3.4]{toms.winter-vi.algs.}.

On the other hand, assuming $A$ has a unique tracial state, it follows from \cite[Theorem 1.1]{niu-mean.dimension.} that $A$ has slow dimension growth. There is a simpler proof for $\V$I algebras, which we omit to keep the exposition at a reasonable length. Therefore, \cite[Theorem 2]{bdr-real.rank.of.inductive.limits.} implies that $A$ has real rank zero.
\end{proof}

%
%
%
%
%

\section{Villadsen Algebras of the second type}\label{sec-the.second.type.}
In this section we study the Villadsen algebras of the second type. We prove that for each Villadsen algebra $A$ of the second type, $F(A)$ has at least one character. For the convenience of the reader we recall the construction from \cite{villadsen-sr.}
\begin{mindef}
Let $X,Y$ be compact Hausdorff spaces. A ${}^\ast$-homomorphism  $\fhi\colon C(X)\ten\K\to C(Y)\ten\K$ is said to be a \textbf{diagonal map of the second type} if there exists $k\in\N$, continuous maps $\lambda_1,\dots,\lambda_k\colon Y\to X$, and mutually orthogonal projections $p_1,\dots,p_k\in C(Y)\ten\K$ such that 
\[
\fhi=(\text{id}_{C(Y)}\ten \alpha)\circ(\tilde\fhi\ten \text{id}_{\K}),
\]
where $\alpha\colon \K\ten\K\to\K$ is some isomorphism and $\tilde\fhi\colon C(X)\to C(Y)\ten\K$ is given by
\[
\tilde\fhi(f)=\sum_{i=1}^k(f\circ\lambda_i)p_i.
\]
In this case, we say $\fhi$ arises from the tuple $(\lambda_i,p_i)_{i=1}^k$, and the maps $\lambda_i$, $i=1,\dots,k$, are referred to as the \textbf{eigenvalue maps} of $\fhi$.
\end{mindef}
Note that in the above definition we have implicitly used that the $C^\ast$-algebra $C(X)\ten\K$ has a natural $C(X)$-module structure. Since all diagonal maps appearing from this point on will be of the second type defined above, we simply refer to them as diagonal maps.

For each $l\in\N$, let $\C P^l$ denote the $l$'th complex projective space,
let $\gamma_l$ denote the universal line bundle over $\C P^l$, and let $\D^l$ denote the $l$-fold cartesian product of the unit disc $\D\subseteq \C$. It is well-known that the $l$-fold cup product $e(\gamma_l)^l$ of the Euler class $e(\gamma_l)$ is non-zero for all $l\in\N$. For each integer $n\geq 1$, let $\sigma(n):=n(n!)$ and $\sigma(0):=1$. Furthermore, let 
$\N_\infty=\N\cup\{\infty\}$ and let $\kappa:\N_\infty\times\N\to\N$ be given by
\[
\kappa(k,n)=\begin{cases}
k\sigma(n), & \text{if }k<\infty,\\
n\sigma(n), & \text{if }k=\infty.\\
\end{cases}
\]
For all integers $k\geq 1$ and $n\geq 0$, define compact Hausdorff spaces
$X^{(k)}_n$ by $X^{(k)}_0:=\D^k$ and
\[
X^{(k)}_n:=\D^k\times \C P^{\kappa(k,1)}\times \C P^{\kappa(k,2)}
\times\cdots\times \C P^{\kappa(k,n)},
\]
when $n\geq 1$. Also, for $k=\infty$, we set $X^{(k)}_0:=\D$ and
\[
X^{(k)}_n:=\D^{n\sigma(n)^2}\times \C P^{\kappa(k,1)}\times \C P^{\kappa(k,2)}
\times\cdots\times \C P^{\kappa(k,n)}.
\]
Thus \( X^{(k)}_n=X^{(k)}_{n-1}\times \C P^{k\sigma(n)} \), whenever $k<\infty$ and $n\geq 1$, and
\begin{align*}
X^{(\infty)}_1&:=X^{(\infty)}_0\times\C P^{1};\\
X^{(\infty)}_n&:=\D^{n\sigma(n)^2-(n-1)\sigma(n-1)^2}\times X^{(\infty)}_{n-1}\times \C
P^{n\sigma(n)},\quad n\geq 2.
\end{align*}
For each $k\in\N_\infty$ and $n\in\N$, let
\begin{align*}
\pi_{k,n}^1:X^{(k)}_n\to X^{(k)}_{n-1},\qquad  \pi_{k,n}^2:X^{(k)}_n\to \C P^{\kappa(k,n)},
\end{align*}
denote the coordinate projections, and set
$\zeta^{(k)}_n:=\pi_{k,n}^{2*}(\gamma_{\kappa(k,n)})$. If $y_0\in X^{(k)}_n$ is a point, we also let $y_0$ denote the constant map $f\colon X^{(k)}_{n+1}\to X^{(k)}_n$ with $f(x)=y_0$ for all $x\in X^{(k)}_{n+1}$.

For each
$k\in\N_\infty$ and integer $n\geq 0$, let $\tilde\fhi^{(k)}_n\colon C(X^{(k)}_n)\ten\K\to
C(X^{(k)}_{n+1})\ten\K$ be the diagonal map arising from the tuple
$(\pi_{k,n+1}^1,\theta_1)\cup(y^{(k)}_{n,j},\zeta^{(k)}_{n+1})_{j=1}^{n+1}$,
where the points $\{y^{(k)}_{n,j}\}_{j=1}^{n+1}\subseteq X^{(k)}_n$ are chosen such that the resulting $C^*$-algebra is
simple (see \cite{villadsen-sr.} for more details) and $\theta_1$ denotes the trivial line bundle. Let $p_0^{(k)}\in C(X^{(k)}_0)\ten\K$ denote a constant projection of rank $1$ and $p^{(k)}_n:=\tilde\fhi^{(k)}_{n,0}(p^{(k)}_0)$.
Furthermore, let
$$A^{(k)}_n:=p^{(k)}_n\big(C(X^{(k)}_n)\ten\K\big)p^{(k)}_n,$$
and $\fhi^{(k)}_n:=\tilde\fhi^{(k)}_n|_{A^{(k)}_n}$. Define $\V_k$ to
be the inductive limit of the system \( (A^{(k)}_n,\fhi^{(k)}_n) \).
The following results about the $C^\ast$-algebras $\mathcal V_k$ may be found in \cite{villadsen-sr.}.

\begin{theo}[Villadsen]\label{theo-villadsens.resultater.}
For each $k\in\N_\infty$, let $\V_k$ be defined as above.
\begin{enumerate}
\item The $C^\ast$-algebra $\V_k$ has a unique tracial state $\tau$, for each $k\in\N_\infty$.
\item The stable rank $\mathrm{sr}(\V_k)$ of $\V_k$ is $k+1$, when $k<\infty$, and infinite, when $k=\infty$.
\item The real rank $\mathrm{RR}(\V_k)$ of $\V_k$ satisfies $k\leq \mathrm{RR}(\V_k)\leq k+1$, when $k<\infty$, and is infinite, when $k=\infty$.
\end{enumerate}
\end{theo}
It is easy to check that, if $\eta$ is an arbitrary vector bundle over $X_i^{(k)}$, then
\begin{align}
\big(\fhi_i^{(k)}\big)^\ast(\eta)\cong \pi_{k,i+1}^{1\ast}(\eta)\oplus (i+1)\text{rank}(\eta)\zeta^{(k)}_{i+1},\label{eq-con.maps.}
\end{align}
where $(\fhi^{(k)}_i)^\ast$ denotes the map from (isomorphism classes of) vector bundles over $X_i^{(k)}$ to (isomorphism classes of) vector bundles over $X_{i+1}^{(k)}$ induced by $\fhi^{(k)}_i$. For each $k,n\in\N$ let $\xi^{(k)}_i$ denote the vector bundle over $X^{(k)}_i$ corresponding to $p^{(k)}_i$. Then \eqref{eq-con.maps.} implies that
\begin{align}
\xi_i^{(k)}\cong \theta_1\times \sigma(1)\gamma_{\kappa(k,1)}\times\cdots\times\sigma(i)\gamma_{\kappa(k,i)}.\label{eq-unit.for.k.}
\end{align}

A brief word on notation: as before, for each $i<j$ and $k\in\N_\infty$, we let 
$\fhi^{(k)}_{i,j}\colon A^{(k)}_i\to A^{(k)}_j$ and $\fhi^{(k)}_{i,\infty}\colon A^{(k)}_i\to \V_k$ denote the induced maps from the inductive limit decomposition. We will often omit the superscript $(k)$ in the following (whenever $k$ is implied by the context).

%
%
\begin{prop}\label{prop-comparability}
Let $k\in\N_\infty$ be given. For each $n\in\N$ there exist projections 
$e_n, q^{(n)}_1,\dots,q^{(n)}_n\in \V_k\ten \K$ such that
\begin{enumerate}
\item $e_n\precsim q^{(n)}_i\oplus q^{(n)}_i$, for all $i=1,\dots,n$.
\item $e_n\not\precsim q^{(n)}_1\oplus\cdots\oplus q^{(n)}_n$.
\item $\displaystyle{ \tau(q_1^{(n)}\oplus q^{(n)}_2\oplus\dots\oplus q^{(n)}_n)\to k}$ and $\tau(e_n)\to 0$ as $n\to\infty$.
\end{enumerate}
\end{prop}
\begin{proof}
We fix an arbitrary $k\in\N_\infty$, and omit $k$ from our notation. 
For each $l\in\N$ and $j=1,\dots,l$, let $\rho^l_j\colon X_l=X^{(k)}_l\to \C P^{\kappa(k,j)}$ denote the coordinate projection. Note that $\rho^l_l=\pi^{2}_{k,l}$ and $\rho^l_j\circ\pi^1_{k,l+1}=\rho^{l+1}_j$ for all $l\geq 1$ and $1\leq j\leq l$. For each $n\in\N$ and $i=1,\dots,n$, let 
$\overline q^{(n)}_i\in A_n\ten\K$ denote the projection corresponding to the vector bundle
\( \eta_{n,i}:= \rho_{i}^{n*}(\kappa(k,i)\cdot\gamma_{\kappa(k,i)})\) over $X_n$, where $\gamma_{\kappa(k,i)}$ is as defined above, and $r_n\in A_n\ten\K$ denote the projection corresponding to the trivial line bundle $\theta_1$. Let $q^{(n)}_i:=\fhi_{n,\infty}(\overline q^{(n)}_i)$ and $e_n:=\fhi_{n,\infty}(r_n)$. We prove that the projections $e_n,q^{(n)}_1,\dots,q^{(n)}_n$ have the properties claimed in the above statement. In the interest of brevity, let
\[
\eta_n:=\eta_{n,1}\oplus \eta_{n,2}\oplus\dots\oplus\eta_{n,n}.
\]
It follows from the Künneth formula, and the fact $e(\gamma_l)^l\neq 0$ for all $l\in\N$, that the Euler class $e(\eta_n)\in H^\ast(X_n)$ is non-zero for each $n\in\N$.

(i): It suffices to prove that $2\kappa(k,i)\cdot\gamma_{\kappa(k,i)}$ dominates a
trivial line bundle for each $i\in\N$. However, this follows from
straightforward dimension considerations. Indeed, since
\begin{align*}
2\cdot\text{rank}(2\kappa(k,i)\cdot\gamma_{\kappa(k,i)})-1\geq 2\kappa(k,i)=\text{dim}(\C P^{\kappa(k,i)}),
\end{align*}
the desired result follows (see for instance \cite[Proposition 9.1.1]{husemoller}).

(ii): Note that it follows from \eqref{eq-con.maps.} that
\[
\fhi_l^\ast(\eta_l)\cong \big(\bigoplus_{j=1}^l\kappa(k,j)\cdot\rho^{(l+1)\ast}_j(\gamma_{\kappa(k,j)})\big)\oplus (l+1)\text{rank}(\eta_l)\cdot\rho^{(l+1)\ast}_{l+1}(\gamma_{\kappa(k,l+1)}).
\]
Since \((l+1)\text{rank}(\eta_l)=(l+1)\sum_{i=1}^l \kappa(k,i)\leq \kappa(k,l+1)\), it follows that $\fhi_l^\ast(\eta_l)\precsim \eta_{l+1}$. By induction, $\fhi_{l,m}^\ast(\eta_l)\precsim \eta_m$, for all $m\geq l$.
Furthermore, again by \eqref{eq-con.maps.}, we have that $\theta_1\precsim\fhi_{l,m}^\ast(\theta_1)$. 

Now, assume that $e_n\precsim q^{(n)}_1\oplus\cdots\oplus q^{(n)}_n$. Since $e_n$ is compact in $\Cu(\V_k)$ it follows from continuity of $\mathbf{Cu}(-)$ that there exists some $m>n$ such that
\[
\theta_1\precsim \fhi_{n,m}^\ast(\theta_1)\precsim \fhi_{n,m}^\ast(\eta_n)\precsim\eta_m.
\]
But since the Euler class of the right hand side is non-zero, this is a contradiction.

(iii): Recall that $\xi_n$ denotes the vector bundle over $X_n$ corresponding to the unit $p_n\in A_n$. Since each $q^{(n)}_i$ is a projection and $\fhi_{i,\infty}$ is unital, we have 
\begin{align*}
\tau(q_1^{(n)}\oplus q^{(n)}_2\oplus\dots\oplus q^{(n)}_n)&=
\frac{ \text{rank}(\eta_n) }{ \text{rank}(\xi_n) }
=\frac{\sum_{l=1}^n \kappa(k,l)}{\sum_{l=0}^n\sigma(l)}\\
&=\frac{\sum_{l=1}^n \kappa(k,l)}{(n+1)!}.
\end{align*}
Hence, when $k<\infty$,
\[
\tau(q_1^{(n)}\oplus q_2^{(n)}\oplus\dots\oplus q^{(n)}_n)=\frac{\big(k\sum_{l=0}^n\sigma(l)\big)-k}{(n+1)!}=\frac{k(n+1)!-k}{(n+1)!}\to k,
\]
while the case $k=\infty$ follows from the observation that
\[
\frac{\sum_{l=1}^n l\sigma(l)}{(n+1)!}\geq\frac{n\sigma(n)}{(n+1)!}=\frac{n^2}{(n+1)} \to \infty.
\]
Similarly, for arbitrary $1\leq k \leq \infty$ we find that
\(\tau(e_n)=\frac{1}{(n+1)!}\to 0\).
\end{proof}

\begin{cor}\label{cor-type.II.characters.}
For each $1\leq k \leq \infty$, the central sequence algebra $F(\mathcal V_k)$ has at least one character.
\end{cor}
\begin{proof}
This is a straightforward consequence of Proposition \ref{prop-comparability} parts (i) and (ii) and Proposition \ref{cor-skarp.kr.} (with $m=2$).
\end{proof}

\begin{remark}
An alternative, albeit slightly artificial, statement of the above corollary is that for each $1\leq k\leq \infty$, the $k$'th Villadsen algebra of the second type $\mathcal V_k$ absorbs $\JS$ if and only if $F(\mathcal V_k)$ has no characters. Indeed, it follows from \cite[Proposition 11]{villadsen-sr.} that $K_0(\mathcal V_k)$ is not weakly unperforated for any $1\leq k<\infty$, and essentially the same proof applies to $k=\infty$. Hence \cite[Theorem 1]{gong.jiang.su-obstructions.to.Z.stability.} implies that $\mathcal V_k\ten\JS\not\cong \mathcal V_k$ for each $1\leq k\leq \infty$. As stated in the above corollary, $F(\mathcal V_k)$ has at least one character for each $1\leq k\leq \infty$, whence the desired result follows.
\end{remark}

\begin{remark}
It was proven in \cite{kirchberg.rordam-char.} that if $A$ is a unital $C^\ast$-algebra with $T(A)\neq\emptyset$ and property (SI), then $F(A)$ has a character if and only if $F(A)/(F(A)\cap J(A))$ has a character (see \cite{kirchberg.rordam-char.} for a definition of $J(A)$). It follows from the above corollary that this is no longer true, if the assumption of property (SI) is removed.
Indeed, let $1\leq k\leq\infty$ be arbitrary and $\mathscr N_k$ denote the weak closure of $\pi_\tau(\V_k)\subseteq B(\mathcal H_\tau)$, where $\pi_\tau$ denotes the GNS representation of $\V_k$ with respect to the tracial state $\tau$. Since $\V_k$ has a unique tracial state, it is a straightforward consequence of \cite[Lemma 2.1]{sato-discrete.amenable.actions.} that
\[
F(\V_k)/(F(\V_k)\cap J(\V_k))\cong \mathscr N_k^\omega\cap \mathscr N_k'.
\]
Here $\mathscr N_k^\omega$ denotes the von Neumann ultrapower of $\mathscr N_k$ with respect to the tracial state $\tau$. Since $\mathscr N_k$ is an injective $\mathrm{II}_1$-factor, it follows that $\mathscr N_k\cong \mathcal R$, where $\mathcal R$ denotes the hyperfinite $\mathrm{II}_1$ factor. In particular, there exists a unital embedding $\mathcal R\to F(\V_k)/(F(\V_k)\cap J(\V_k))$ whence $F(\V_k)/(F(\V_k)\cap J(\V_k))$ does not have any characters. Hence, the above corollary shows that the assumption of property (SI) in \cite[Proposition 3.19]{kirchberg.rordam-char.} is indeed necessary.
\end{remark}

Proposition \ref{prop-comparability} (iii) allows us to compute the radius of comparison for each $\V_k$ (the radius of comparison was originally defined by Toms in \cite{toms-flat.dimension.growth}, and an extended definition was given in \cite{rc-algebraic.} and shown to agree with the original definition for all sufficiently finite $C^\ast$-algebras, e.g., unital, simple and stably finite $C^\ast$-algebras).
\begin{cor}\label{cor-VI.II.rc.}
$\textup{rc}(\V_k)=k$ for each $1\leq k<\infty$.
\end{cor}
\begin{proof}
Fix $1\leq k<\infty$. By \cite[Corollary 5.2]{toms-smooth.min.} and \cite[Proposition 3.2.4]{rc-algebraic.} 
\[
\text{rc}(\V_k)\leq \lim_{n\to\infty}\frac{\text{dim}(X^{(k)}_n)}{2\cdot\text{rank}(p^{(k)}_n)}= \lim_{n\to\infty}\frac{k(n+1)!}{(n+1)!}=k.
\] 
Fix arbitrary $\e>0$. By Proposition \ref{prop-comparability} parts (ii) and (iii) we may choose projections $e, q\in \V_k\ten\K$ such that $\tau(e)<\e/2$, $\tau(q)>k-\e/2$, while $e\not\precsim q$. In particular $d_\tau(e)+(k-\e)=\tau(e)+
k-\e<k-\e/2 < d_\tau(q)$, while $e\not\precsim q$. Since $\e>0$ was arbitrary, it follows that $\text{rc}(\V_k)\geq k$.
\end{proof}

The proof of the above corollary can easily be modified to show that $\text{rc}(\V_\infty)=\infty$ (or even that $r_{\V_\infty,\infty}=\infty$; see \cite{rc-algebraic.} for a definition of $r_{\V_\infty,\infty}$), but as evidenced by Theorem \ref{theo-non.cfp.}, in this case a stronger statement holds

%
%
%
%
%
%

\appendix
\section{Appendix: The failure of the Corona Factorization Property for the Villadsen algebra $\V_\infty$}

\qquad{\large By Joan Bosa\footnote{School of Mathematics and Statistics, University of Glasgow, 15 University Gardens, G12 8QW, Glasgow, UK. {\it E-mail address: joan.bosa@glasgow.ac.uk} } and Martin S. Christensen}\\
\newline
In this appendix we prove that the Villadsen algebra $\V_\infty$ does not satisfy the Corona Factorization Property (CFP), thereby improving the result, from an earlier version of this paper, that $\V_\infty$ does not satisfy the $\omega$-comparison property.

Both $\omega$-comparison and the CFP may be regarded as comparison properties of the Cuntz semigroup invariant, and both properties are related to the question of when a given C*-algebra is stable (see for instance \cite[Proposition 4.8]{opr-cfp.}). In particular, a simple, separable $C^\ast$-algebra $A$ has the {\bf CFP} if and only if, whenever $x,y_1,y_2,\dots$ are elements in $\Cu(A)$ and $m\geq 1$ is an integer satisfying $x\leq my_j$ for all $j\geq 1$, then $x\leq \sum_{i=1}^\infty y_i$ (\cite[Theorem 5.13]{opr-cfp.}). On the other hand, given a simple $C^\ast$-algebra $A$, $\Cu(A)$ has {\bf $\omega$-comparison} if and only if $\infty=x\in\Cu(A)$ whenever $f(x)=\infty$ for all functionals $f$ on $\Cu(A)$ (\cite[Proposition 5.5]{bosa.petzka-comparison.}). Recall that a \textit{functional} $f$ on the Cuntz semigroup $\Cu(A)$ of a $C^\ast$-algebra $A$ is an ordered semigroup map $f\colon \Cu(A)\to[0,\infty]$ which preserves suprema of increasing sequences. In particular, the latter comparability condition is satisfied for all unital $C^\ast$-algebras $A$ with finite radius of comparison by \cite{rc-algebraic.}.

From the above characterization it follows that any separable $C^\ast$-algebra $A$ whose Cuntz semigroup $\Cu(A)$ has the $\omega$-comparison property also has the CFP (see \cite[Proposition 2.17]{opr-cfp.}). Whether the converse implication is true remains an open question.
This question was considered by the first author of this appendix and Petzka in \cite{bosa.petzka-comparison.}, where the failure of the converse implication was shown just in the algebraic framework of the category $\Cu$. However, it was emphasized there that a more analytical approach will be necessary in order to verify (or disprove) the converse implication for any (simple) $C^\ast$-algebra $A$.

The Villadsen algebras have been used several times to certify bizarre behaviour in the theory of C*-algebras; hence, after Ng and Kucerosvky showed in \cite{kucerosvky.ng-perforation.and.cfp.} that $\mathcal V_2$ satisfies the CFP, one wonders whether it satisfies the $\omega$-comparison or not. From Corollary \ref{cor-VI.II.rc.} (together with \cite[Theorem 4.2.1]{rc-algebraic.}) one gets that, for all $1\leq n<\infty$, the Cuntz semigroups $\Cu(\V_n)$ have the $\omega$-comparison property and hence the CFP. But this is not the case for the $C^\ast$-algebra $\V_\infty$. As demonstrated below, it does not have the CFP (and hence $\Cu(\V_\infty)$ does not have $\omega$-comparison). Notice that although $\V_\infty$ has a different structure than $\V_n$, it does not witness the potential non-equivalence of $\omega$-comparison and the CFP for unital, simple and stably finite $C^\ast$-algebras.


\begin{theo}\label{theo-non.cfp.}
Let $\V_\infty$ be given as above. Then $\V_\infty$ is a unital, simple, separable and nuclear $C^\ast$-algebras with a unique tracial state such that the Cuntz semigroup $\Cu(\V_\infty)$ does not have the Corona Factorization Property for semigroups.
\end{theo}
\begin{proof}
We use the notation introduced above, with $k=\infty$ fixed and omitted. Additionally, for each $n\geq 1$, let $\lambda(n):=\kappa(\infty,n)=n\sigma(n)=n^2(n!)$, and $Y_n:=\C P^{\lambda(1)}\times\cdots\times \C P^{\lambda(n)}$. Note that $X_n=\D^{n\sigma(n)^2}\times Y_n$, let $\overline\pi_n\colon X_n\to Y_n$ denote the coordinate projection and $\overline\psi_n\colon C(Y_n)\ten\K\to C(X_n)\ten\K\cong A_n\ten\K$ denote the ${}^\ast$-homomorphism induced by $\overline\pi$. 

For each $n\geq 1$ and $1\leq j\leq n$, let $\rho_{n,j}\colon Y_n\to\C P^{\lambda(j)}$ denote the projection map and let $\overline\zeta_{n,j}$ denote the vector bundle $\rho_{n,j}^\ast(\gamma_{\lambda(j)})$ over $Y_n$. To avoid overly cumbersome notation, we simply write $\overline \zeta_j$ for $\overline\zeta_{n,j}$ whenever $j\leq n$. Furthermore, for each $n\geq 1$, let $\overline\xi_n$ denote the vector bundle over $Y_n$ given by 
\(\theta_1\oplus\sigma(1)\overline \zeta_1\oplus\cdots\oplus \sigma(n)\overline\zeta_n\). Recall that, for each $j\geq 1$, $\zeta_j$ denotes the vector bundle $\pi^{2\ast}_j(\gamma_{\lambda(j)})$ over $X_j$. To avoid overly cumbersome notation we also let $\zeta_j$ denote the vector bundle $\pi^{1\ast}_{n}\circ\cdots\circ\pi^{1\ast}_{j+1}(\zeta_j)$, whenever $n>j$. With this notation, the vector bundle $\xi_n$ over $X_n$ corresponding to the unit $p_n\in A_n$ may be written 
\(\xi_n\cong \theta_1\oplus \sigma(1)\zeta_1\oplus\cdots\oplus\sigma(n)\zeta_n\).
It is immediately verified that $\overline\psi_n(\overline\zeta_j)\cong \zeta_j$ for all $j\leq n$, and in particular $\psi_n^\ast(\overline\xi_n)\cong \xi_n$. Hence, if $q\in C(Y_n)\ten\K$ is a projection corresponding to a vector bundle $\eta$ satisfying $\overline\xi_n\precsim \eta$, then $p_n\precsim \psi_n(q)$.

Note that, 
\begin{align}
\lim_{n\to\infty}\frac{\textup{dim}(Y_n)}{2\lambda(n)}\leq\lim_{n\to\infty}\frac{n^2(n!)+n\big(\sum_{i=0}^{n-1}\sigma(i)\big)}{n^2(n!)}=1+\lim_{n\to\infty}\frac{1}{n}=1.\label{eq-dimension.growth.}
\end{align}
Furthermore, it follows from \eqref{eq-con.maps.}, by induction, that for any $1\leq n< m$ and an arbitary vector bundle $\eta$ over $X_n$, we have
\begin{align}
\fhi_{n,m}^\ast(\eta)\cong\mu_{m,n}^{\ast}(\eta)\oplus (n+1)\text{rank}(\eta)\zeta_{n+1}
\oplus\cdots\oplus
\frac{\sigma(m)}{(n+1)!}\text{rank}(\eta)\zeta_m,\label{eq-con.maps.2.}
\end{align}
where $\mu_{m,n}:=\pi_{n+1}^1\circ\cdots\circ \pi_m^1\colon X_m\to X_n$. Moreover, we find that
\begin{align}
\lim_{m\to\infty}\frac{\sigma(m)}{(n+1)!\lambda(m)}=\lim_{m\to\infty}\frac{1}{(n+1)!m}=0.\label{eq-con.maps.growth.}
\end{align}
Choose $l(1)\geq 1$ large enough that $\lambda(k)$ is divisible by $4$ and $\frac{5}{4}\lambda(k)\geq \textup{dim}(Y_k)/2>\textup{rc}(C(Y_k)\ten\K)$ for all $k\geq l(1)$, which is possible by \eqref{eq-dimension.growth.}. Set $k(1):=\frac{1}{2}\lambda(l(1))$. Define sequences $(l(n))_{n\geq 1}$ and $(k(n))_{n\geq 1}$ as follows: given $n\geq 2$ and $l(1),\dots,l(n-1)$ choose $l(n)>l(n-1)$ such that 
\begin{align}
\frac{\big(\sum_{j=1}^{l(n-1)}\lambda(j)\big)\sigma(l(n))}{(l(n-1)+1)!\lambda(l(n))}\leq \frac 12,\text{ i.e., }\frac{\big(\sum_{j=1}^{l(n-1)}\lambda(j)\big)\sigma(l(n))}{(l(n-1)+1)!}\leq \frac {\lambda(l(n))}{2},\label{eq-fast.enough.growth}
\end{align}
which is possible by \eqref{eq-con.maps.growth.}, and set $k(n):=\frac12\lambda(l(n))$. Finally, for each $n\geq 1$, let $\overline q_n\in A_{l(n)}\ten\K$ be the projection corresponding to the vector bundle $\zeta_{l(n)}$ over $X_{l(n)}$ and $x_n:=k(n)\langle \fhi_{l(n),\infty}(q_n) \rangle\in \Cu(\V_\infty)$. We aim to show that the sequence $(x_n)_{n\geq1}$ in $\Cu(\V_\infty)$ witnesses the failure of the Corona Factorization Property in $\Cu(\V_\infty)$.

First, we show that $5x_n\geq \langle \mathbf 1_{\V}\rangle=:e$ for all $n\geq 1$. As noted above, it suffices to show that $5k(n)\overline\zeta_{\lambda(l(n))}\geq \overline \xi_{\lambda(l(n))}$. But, by choice of $k(n)$ and $l(n)$ we have that
\[
\rank(5k(n)\overline\zeta_{\lambda(l(n))})=\frac 52\lambda(l(n))\geq \frac{\textup{dim}(Y_{l(n)})}2+\rank(\overline\xi_{l(n)}),
\]
since $\rank(\overline\xi_{l(n)})=(l(n)+1)!\leq \frac{\textup{dim}(Y_{l(n)})}{2}$. The desired result therefore follows from \cite[Theorem 2.5]{goodearl-riesz.decomposition.}.

Next, we show that $e\not\leq\sum_{i=1}^\infty x_i$. Proceeding as in the proof of Proposition \ref{prop-comparability} part (ii), it suffices to prove that 
\[
\langle p_j\rangle\not\leq \langle  \bigoplus_{i=1}^n\fhi_{l(i),j}(q_i) \rangle
\]
for all $j\geq l(n)$ (recall that $p_j\in A_j$ denotes the unit, i.e., the projection corresponding to $\xi_j$). Since $\xi_j$ dominates a trivial line bundle for each $j$, it suffices to prove that the vector bundle corresponding to the right hand side above does not. We do this by proving that
\begin{align}
\bigoplus_{i=1}^n \fhi_{l(i),j}^\ast(k(i)\zeta_{l(i)})\precsim \bigoplus_{s=1}^j\lambda(s)\zeta_s.\label{eq-technical.thing.}
\end{align}
Since the right hand side does not dominate any trivial bundle, by the proof of Proposition \ref{prop-comparability} part (ii), this will complete the proof. Note that it also follows from the proof of Proposition \ref{prop-comparability} part (ii) that \(\fhi_{j,m}^\ast(\bigoplus_{s=1}^j\lambda(s)\zeta_s)\precsim \bigoplus_{s=1}^m\lambda(s)\zeta_s\) for all $m\geq j$. Thus, it suffices to prove that
\[
\bigoplus_{i=1}^{n-1} \fhi_{l(i),l(n)}^\ast(k(i)\zeta_{l(i)})\oplus k(n)\zeta_{l(n)}\precsim \bigoplus_{s=1}^{l(n)}\lambda(s)\zeta_s
\]
for all $n\geq 1$. We proceed by induction. Clearly the statement is true for $n=1$, so suppose it is true for $n-1$ with $n\geq 2$. Then 
\[
\bigoplus_{i=1}^{n-1} \fhi_{l(i),l(n)}^\ast(k(i)\zeta_{l(i)})\oplus k(n)\zeta_{l(n)}\precsim\fhi^\ast_{l(n-1),l(n)}\big(\bigoplus_{s=1}^{l(n-1)}\lambda(s)\zeta_s\big)\oplus k(n)\zeta_{l(n)}.
\]
by induction hypothesis.

Now, letting $N:=\sum_{s=1}^{l(n-1)}\lambda(s)=\rank(\bigoplus_{s=1}^{l(n-1)}\lambda(s)\zeta_s)$, it follows by the choice of $l(n)$ and $k(n)$  (see \eqref{eq-fast.enough.growth}) that $k(n)+\frac{N\sigma(l(n))}{(l(n-1)+1)!}\leq \lambda(l(n))$. Hence, combining the above induction step with \eqref{eq-con.maps.2.}, one has:
 \begin{align*}
 &\bigoplus_{i=1}^{n-1} \fhi_{l(i),l(n)}^\ast(k(i)\zeta_{l(i)})\oplus k(n)\zeta_{l(n)} \\
&  \precsim \fhi_{l(n-1),l(n)}^\ast\big(\bigoplus_{s=1}^{l(n-1)}\lambda(s)\zeta_s\big)\oplus k(n)\zeta_{l(n)}\\
&  \overset{\eqref{eq-con.maps.2.}}{\precsim} \big(\bigoplus_{s=1}^{l(n-1)}\lambda(s)\zeta_s\big)\oplus(l(n-1)+1)N\zeta_{l(n-1)+1}\\
&\qquad\oplus\cdots\oplus\frac{N\sigma(l(n))}{(l(n-1)+1)!}\zeta_{l(n)}\oplus k(n)\zeta_{l(n)}\\ 
&\precsim \bigoplus_{s=1}^{l(n)}\lambda(s)\zeta_s.
 \end{align*}
Thus, the desired result follows.
\end{proof}


\begin{thebibliography}{10}
\small 
\parskip=.5ex

\bibitem{kirchberg.hiroshi-non.commutative.central.sequence.}
H.~Ando and E.~Kirchberg.
\newblock Non-commutativity of the central sequence algebra for separable
  non-type {I} {$C^*$}-algebras.
\newblock {\em J. Lond. Math. Soc. (2)}, 94(1):280--294, 2016.

\bibitem{APT-the.cuntz.semigroup.}
P.~Ara, F.~Perera, and A.~S. Toms.
\newblock {$K$}-theory for operator algebras. {C}lassification of {$C\sp
  *$}-algebras.
\newblock In {\em Aspects of operator algebras and applications}, volume 534 of
  {\em Contemp. Math.}, pages 1--71. Amer. Math. Soc., Providence, RI, 2011.

\bibitem{bdr-real.rank.of.inductive.limits.}
B.~Blackadar, M.~D{\u{a}}d{\u{a}}rlat, and M.~R{\o}rdam.
\newblock The real rank of inductive limit {$C\sp *$}-algebras.
\newblock {\em Math. Scand.}, 69(2):211--216 (1992), 1991.

\bibitem{rc-algebraic.}
B.~Blackadar, L.~Robert, A.~P. Tikuisis, A.~S. Toms, and W.~Winter.
\newblock An algebraic approach to the radius of comparison.
\newblock {\em Trans. Amer. Math. Soc.}, 364(7):3657--3674, 2012.

\bibitem{bosa.petzka-comparison.}
J.~Bosa and H.~Petzka.
\newblock Comparison properties of the cuntz semigroup and applications to
  {$C\sp{\ast}$}-algebras.
\newblock 2016.
\newblock {\em Canad. J. of Maths.}, to appear. arXiv:1603.07199.

\bibitem{CEI-the.cuntz.semigroup.}
K.~T. Coward, G.~A. Elliott, and C.~Ivanescu.
\newblock The {C}untz semigroup as an invariant for {$C\sp *$}-algebras.
\newblock {\em J. Reine Angew. Math.}, 623:161--193, 2008.

\bibitem{elliott.gong.lin.niu-classification.}
G.~A. Elliott, G.~Gong, H.~Lin, and Z.~Niu.
\newblock On the classification of simple amenable {$C\sp\ast$}-algebras with
  finite decompostion rank {II}, preprint.
\newblock 2015.
\newblock arXiv:1507.03437.

\bibitem{elliott.niu-zero.mean.dimension.}
G.~A. Elliott and Z.~Niu.
\newblock The {$C\sp\ast$}-algebra of a minimal homeomorphism of zero mean
  dimension, preprint.
\newblock 2014.
\newblock arXiv:1406.2382v2.

\bibitem{farah.hart.sherman-model.theory.I.}
I.~Farah, B.~Hart, and D.~Sherman.
\newblock Model theory of operator algebras {I}: stability.
\newblock {\em Bull. Lond. Math. Soc.}, 45(4):825--838, 2013.

\bibitem{ge.hadwin-ultraproducts.}
L.~Ge and D.~Hadwin.
\newblock Ultraproducts of {$C\sp *$}-algebras.
\newblock In {\em Recent advances in operator theory and related topics
  ({S}zeged, 1999)}, volume 127 of {\em Oper. Theory Adv. Appl.}, pages
  305--326. Birkh\"auser, Basel, 2001.

\bibitem{giol.kerr-perforation.}
J.~Giol and D.~Kerr.
\newblock Subshifts and perforation.
\newblock {\em J. Reine Angew. Math.}, 639:107--119, 2010.

\bibitem{gong.jiang.su-obstructions.to.Z.stability.}
G.~Gong, X.~Jiang, and H.~Su.
\newblock Obstructions to {$\mathcal Z$}-stability for unital simple {$C\sp
  *$}-algebras.
\newblock {\em Canad. Math. Bull.}, 43(4):418--426, 2000.

\bibitem{gong.lin.niu-classification.of.finite.simple.algebras.}
G.~Gong, H.~Lin, and Z.~Niu.
\newblock Classification of finite simple amenable {$\mathcal Z$}-stable
  {$C\sp\ast$}-algebras.
\newblock 2015.
\newblock arXiv:1501.00135.

\bibitem{goodearl-riesz.decomposition.}
K.~R. Goodearl.
\newblock Riesz decomposition in inductive limit {$C^*$}-algebras.
\newblock {\em Rocky Mountain J. Math.}, 24(4):1405--1430, 1994.

\bibitem{husemoller}
D.~Husemoller.
\newblock {\em Fibre bundles}, volume~20 of {\em Graduate Texts in
  Mathematics}.
\newblock Springer-Verlag, New York, third edition, 1994.

\bibitem{kirchberg-central.sequences.}
E.~Kirchberg.
\newblock Central sequences in {$C\sp *$}-algebras and strongly purely infinite
  algebras.
\newblock In {\em Operator {A}lgebras: {T}he {A}bel {S}ymposium 2004}, volume~1
  of {\em Abel Symp.}, pages 175--231. Springer, Berlin, 2006.

\bibitem{kirchberg.rordam-char.}
E.~Kirchberg and M.~R{\o}rdam.
\newblock When central sequence {$C\sp *$}-algebras have characters.
\newblock {\em Internat. J. Math.}, 26(7):1550049, 32, 2015.

\bibitem{kucerovsky.ng-the.cfp.and.approximate.unitary.equivalence.}
D.~Kucerovsky and P.~W. Ng.
\newblock The corona factorization property and approximate unitary
  equivalence.
\newblock {\em Houston J. Math.}, 32(2):531--550 (electronic), 2006.

\bibitem{kucerovsky.ng-the.cfp.}
D.~Kucerovsky and P.~W. Ng.
\newblock {$S$}-regularity and the corona factorization property.
\newblock {\em Math. Scand.}, 99(2):204--216, 2006.

\bibitem{kucerosvky.ng-perforation.and.cfp.}
D.~Kucerovsky and P.~W. Ng.
\newblock A simple {$C\sp \ast$}-algebra with perforation and the corona
  factorization property.
\newblock {\em J. Operator Theory}, 61(2):227--238, 2009.

\bibitem{matui.sato-strict.comparison.}
H.~Matui and Y.~Sato.
\newblock Strict comparison and {$\mathcal{Z}$}-absorption of nuclear {$C\sp
  *$}-algebras.
\newblock {\em Acta Math.}, 209(1):179--196, 2012.

\bibitem{matui.sato-decomposition.rank.}
H.~Matui and Y.~Sato.
\newblock Decomposition rank of {UHF}-absorbing {$C\sp *$}-algebras.
\newblock {\em Duke Math. J.}, 163(14):2687--2708, 2014.

\bibitem{mcduff-central.sequences.}
D.~McDuff.
\newblock Central sequences and the hyperfinite factor.
\newblock {\em Proc. London Math. Soc. (3)}, 21:443--461, 1970.

\bibitem{milnor+stasheff}
J.~W. Milnor and J.~D. Stasheff.
\newblock {\em Characteristic classes}.
\newblock Princeton University Press, Princeton, N. J.; University of Tokyo
  Press, Tokyo, 1974.
\newblock Annals of Mathematics Studies, No. 76.

\bibitem{niu-mean.dimension.}
Z.~Niu.
\newblock Mean dimension and {AH}-algebras with diagonal maps.
\newblock {\em J. Funct. Anal.}, 266(8):4938--4994, 2014.

\bibitem{opr-cfp.}
E.~Ortega, F.~Perera, and M.~R{\o}rdam.
\newblock The corona factorization property, stability, and the {C}untz
  semigroup of a {$C\sp \ast$}-algebra.
\newblock {\em Int. Math. Res. Not. IMRN}, (1):34--66, 2012.

\bibitem{robert.rordam-div.}
L.~Robert and M.~R{\o}rdam.
\newblock Divisibility properties for {$C\sp *$}-algebras.
\newblock {\em Proc. Lond. Math. Soc. (3)}, 106(6):1330--1370, 2013.

\bibitem{rordam-stability.not.stable.property.}
M.~R{\o}rdam.
\newblock Stability of {$C\sp *$}-algebras is not a stable property.
\newblock {\em Doc. Math.}, 2:375--386 (electronic), 1997.

\bibitem{rordam-finite.and.infinite.projection.}
M.~R{\o}rdam.
\newblock A simple {$C\sp *$}-algebra with a finite and an infinite projection.
\newblock {\em Acta Math.}, 191(1):109--142, 2003.

\bibitem{sato-discrete.amenable.actions.}
Y.~Sato.
\newblock Discrete amenable group actions on von neumann algebras and invariant
  nuclear {$C\sp\ast$}-subalgebras, preprint.
\newblock 2011.
\newblock arXiv:1104.4339v1.

\bibitem{tikuisis.white.winter-quasidiagonality.}
A.~Tikuisis, S.~White, and W.~Winter.
\newblock Quasidiagonality of nuclear {$C^\ast$}-algebras.
\newblock {\em Ann. of Math. (2)}, 185(1):229--284, 2017.

\bibitem{toms-counterexample.}
A.~S. Toms.
\newblock On the independence of {$K$}-theory and stable rank for simple {$C\sp
  *$}-algebras.
\newblock {\em J. Reine Angew. Math.}, 578:185--199, 2005.

\bibitem{toms-flat.dimension.growth}
A.~S. Toms.
\newblock Flat dimension growth for {$C\sp *$}-algebras.
\newblock {\em J. Funct. Anal.}, 238(2):678--708, 2006.

\bibitem{toms-the.classification.problem.}
A.~S. Toms.
\newblock On the classification problem for nuclear {$C\sp \ast$}-algebras.
\newblock {\em Ann. of Math. (2)}, 167(3):1029--1044, 2008.

\bibitem{toms-smooth.min.}
A.~S. Toms.
\newblock Comparison theory and smooth minimal {$C\sp *$}-dynamics.
\newblock {\em Comm. Math. Phys.}, 289(2):401--433, 2009.

\bibitem{toms.winter-vi.algs.}
A.~S. Toms and W.~Winter.
\newblock The {E}lliott conjecture for {V}illadsen algebras of the first type.
\newblock {\em J. Funct. Anal.}, 256(5):1311--1340, 2009.

\bibitem{villadsen-perforation.}
J.~Villadsen.
\newblock Simple {$C\sp *$}-algebras with perforation.
\newblock {\em J. Funct. Anal.}, 154(1):110--116, 1998.

\bibitem{villadsen-sr.}
J.~Villadsen.
\newblock On the stable rank of simple {$C\sp \ast$}-algebras.
\newblock {\em J. Amer. Math. Soc.}, 12(4):1091--1102, 1999.

\end{thebibliography}

\bigskip \footnotesize \noindent {\scshape Department of Mathematical Sciences,
University of Copenhagen\\
Universitetsparken 5,
DK-2100 Copenhagen Ø}
\par\medskip\noindent {\itshape E-mail address:} {\ttfamily martin.christensen@math.ku.dk}
\end{document}